\documentclass[11pt, a4paper]{article}
\usepackage{amsmath, amssymb, amsfonts, amsthm}
\usepackage[pdftex]{graphicx, color}
\usepackage{ulem}
\usepackage{xcolor}
\usepackage{hyperref}

\newcommand{\E}{{\bf{E}}}
\newcommand{\PP}{{\bf{P}}}
\newcommand{\Var}{{\bf{Var}}}

\newtheorem{tm}{Theorem}
\newtheorem{lem}{Lemma}
\newtheorem{cor}{Corollary}
\newtheorem{rem}{Remark}

\newcommand{\JKcorr}[1]{\textcolor{red}{#1}}

\hypersetup{
colorlinks=true,       
linkcolor=blue,          
citecolor=blue        
}
\begin{document}

\parindent=0pt

\smallskip
\par\vskip 3.5em
\centerline{\Large \bf Normal and stable approximation to subgraph counts}
\smallskip
\centerline{\Large \bf in superpositions of Bernoulli random graphs}
\vglue2truecm

\vglue1truecm
\centerline{Mindaugas Bloznelis$^{1}$, and
Joona Karjalainen$^{2}$, and Lasse Leskel\"a$^{2}$}

\smallskip

\centerline{$^1$Institute of Informatics, Vilnius University}
\centerline{
\, \ Naugarduko 24, LT-03225 Vilnius, Lithuania} 

\smallskip
	
\centerline{$^2$Department of Mathematics and Systems Analysis}
\centerline{\ \ School of Science, Aalto University}
\centerline{Otakaari 1, FI-02150 Espoo, Finland}

\vglue2truecm

\abstract{
The
 clustering property of  complex networks
 indicates the abundance of 
 small dense subgraphs in otherwise sparse networks. 
 For a community-affiliation network defined by  a superposition of Bernoulli random graphs, which has a  nonvanishing global clustering coefficient and a power-law degree distribution,
 we establish  normal and $\alpha$--stable approximations to the number of small cliques, cycles and more general $2$-connected subgraphs.}

\vglue2truecm


\section{Introduction and results}

Mathematical modeling of complex networks aims at explaining and reproducing characteristic properties of large real-world networks, e.g., power-law degree distributions and clustering. By clustering property we  mean  the tendency of nodes to cluster together by forming relatively small groups with a high density of ties within a group. Locally, in the vicinity of a vertex, clustering can be measured by the local clustering coefficient, the probability that two randomly selected neighbors of the vertex are adjacent. Globally, the fraction of wedges (paths of length $2$) that induce triangles  defines the global clustering coefficient, which represents the  probability that endpoints of a randomly selected wedge are adjacent. Non-vanishing clustering coefficients indicate the abundance of triangles and other small dense subgraphs in otherwise sparse networks. 
 The problem of determining the asymptotic distribution
 of dense subgraph counts in  sparse complex networks
  is of considerable interest, but there are very few results obtained  so far.
 

 In the present study we establish  normal and $\alpha$-stable approximations of the numbers of $k$-cliques, $k$-cycles and more general $2$-connected subgraphs in a community-affiliation network model defined by a superposition of Bernoulli random graphs \cite{Lasse_Mindaugas2019, Yang_Leskovec2012, Yang_Leskovec2014}.
 
 To the best of our knowledge this is the first
 systematic 
 study of an $\alpha$-stable approximation to subgraph counts in a theoretical model of sparse affiliation networks. 
 We note that in the network model considered, the clustering property and the power-law degree distribution, two basic properties of complex networks, are essential  for an $\alpha$-stable limit to emerge. 
   
 \subsection{Network model}
 Let $(X,Q)$ be a random vector with values in $\{0,1,2, \dots\}\times [0,1]$ and let
 \begin{displaymath}
\bigl\{G(x,p): \ \, x\in\{1,2\dots\}, \, p\in [0,1] \bigr\}
\end{displaymath}
 be a family of Bernoulli random graphs independent of $(X,Q)$.  We set $[x]=\{1,2,\dots, x\}$ to be  the vertex set of $G(x,p)$. 
 Recall that in $G(x,p)$ 
 every pair of vertices
 $\{i,j\}\subset [x]$ is declared adjacent independently at random with probability $p$.
 For notational convenience we introduce the empty graph $G_{\emptyset}$ having no vertices and set $G(0,p)=G_{\emptyset}$ for any $p\in [0,1]$. We define the mixture of Bernoulli random graphs $G(X,Q)$ in a natural way. 

\medskip

Let $(X_1,Q_1), (X_2,Q_2),\dots$ be a sequence of independent 
copies of $(X,Q)$. Given  $X_1,\dots$, $X_m$,
 let ${\cal V}_{n,i}={\cal V}_{n,i}(X_i)$, $1\le i\le m$, be 
 independent random subsets of $[n]$ defined as follows. 
 For $X_i\le n$ we select ${\cal V}_{n,i}$ uniformly at random from 
 the class of subsets of $[n]$ of size $X_i$. 
 For $X_i>n$ we set ${\cal V}_{n,i}=[n]$. 
 We denote ${\tilde X}_i=|{\cal V}_{n,i}|=X_i\wedge n$. 
 Let $G_{n,i}$, $1\le i\le m$, be independent random graphs with
 vertex sets ${\cal V}_{n,i}$ defined as follows. We obtain $G_{n,i}$ 
 by \JKcorr{a} one-to-one mapping of vertices of 
 $G({\tilde X}_i,Q_i)$ to the elements of ${\cal V}_{n,i}$ 
 and \JKcorr{by} retaining the adjacency relations of 
 $G({\tilde X}_i,Q_i)$.
 We denote by ${\cal E}_{n,i}$ the edge set of $G_{n,i}$. Finally, let 
 $G_{[n,m]}=(V, {\cal E})$ be the random graph with the vertex set 
 $V=[n]$ and edge set 
 ${\cal E}={\cal E}_{n,1}\cup\cdots\cup {\cal E}_{n,m}$. Therefore 
 $G_{[n,m]}$ is the superposition of 
 $G_{n,1}$, $\dots$, $G_{n,m}$. We call the contributing random 
 graphs $G_{n,i}$ layers or communities.

\medskip
 
 

The random graph $G_{[n,m]}$  represents a null model of a community-affiliation network
 \cite{Yang_Leskovec2012, Yang_Leskovec2014},
 which has attracted considerable attention in the literature.
In the particular case where $Q\equiv 1$ the random graph $G_{[n,m]}$
goes back to 
  the `passive' model of random intersection graph
 \cite{GodehardJaworski2001}. 
In the  parameter 
regime $m=\Theta(n)$ as $m,n\to+\infty$ the random graph $G_{[n,m]}$ admits 
a
power-law degree distribution with tunable power-law exponent,  nonvanishing global clustering coefficient and  
 tunable clustering spectrum \cite{Lasse_Mindaugas2019}.
 Moreover, it admits
 a limiting bidegree distribution with (stochastically dependent) power-law marginals   shown in 
 \cite{Joona_Lasse_Mindaugas2021}.
 The present paper continues the study of the random graph $G_{[n,m]}$ and focuses on the asymptotic distributions of (dense) subgraph counts.

\subsection{Results}

Let $F=({\cal V}_F,{\cal E}_F)$ be a graph with vertex set ${\cal V}_F$ and edge set ${\cal E}_F$. We denote $v_F=|{\cal V}_F|$ and $e_F=|{\cal E}_F|$. We assume in what follows that $F$ is $2$-connected. That is, $F$ is connected and, moreover,  it stays connected even if we remove any one of its vertices. We call $F$ balanced if $e_F/v_F=\max\{e_H/v_H:\, H\subset F$ with $e_H\ge 1\}$. For example, the cycle ${\cal C}_k$  and clique ${\cal K}_k$ (where $k$ stands for the number of vertices) are $2$-connected and balanced. 
Let
 $N_F$ be the number of copies of $F$ in $G(X,Q)$. Denote  
 $\sigma^2_F=\Var N_F$, the variance of $N_F$. We write 
 $\sigma^2_F<\infty$ if the variance is finite and 
 $\sigma^2_F=\infty$ otherwise. 
 We use the shorthand notation 
 $N_F^*:=\E(N_F|X,Q)
 =a_F\binom{X}{v_F}Q^{e_F}$, where $a_F$ stands for 
 the number of distinct copies of $F$ in the complete graph on $v_F$ vertices.
 We have, for example, that
   $N_{{\cal C}_k}^*=(X)_kQ^k/(2k)$  and 
   $N_{{\cal K}_k}^*=(X)_kQ^{\binom{k}{2}}/k!$. 
   Here and below
  $(x)_k=x(x-1)\cdots(x-k+1)$ denotes the falling factorial.

Let ${\cal N}_F$ be the number of copies of  $F$  in $G_{[n,m]}$.
Our first  result establishes the asymptotic normality  of 
${\cal N}_F$ .


 \begin{tm}\label{normal_limit_B}
 Let $\nu>0$. Let $n,m\to+\infty$ and  assume that  $m/n\to\nu$.
  Let $F$ be a $2$-connected graph with $v_F\ge 3$ vertices. 
  Assume that  $\E X<\infty$ and  $0< \sigma^2_{F}<\infty$.
  Assume, in addition, that
  \begin{equation}
\label{2021-06-29}
\E \left(X^{1+s\left(1-\frac{1}{2e_F}\right)}Q^{s}\right)<\infty, \qquad
{\text{for each}}\qquad
 s=1,2,\dots, v_F-1.
\end{equation}
Then 
  $({\cal N}_F-\E {\cal N}_F)/(\sigma_{F}\sqrt{m})$ converges in distribution to  the standard normal distribution. 
\end{tm}

\begin{rem}\label{remark-1}
 For a balanced graph $F$, the finite variance condition $\sigma_F^2<\infty$ is equivalent to the second moment condition $\E(N_F^*)^2<\infty$. 
 In particular, we have  
 \begin{equation}\label{2021-06-23}
 \sigma_F^2<\infty\Leftrightarrow \E(X^{2v_F}Q^{2e_F})<\infty.
\end{equation}
\end{rem}

\begin{rem}\label{remark-clique_normal}
In the special case, where $F$ is a clique on $k\ge 3$ vertices ($F={\cal K}_k$)  condition (\ref{2021-06-29})
can be replaced by the following one
 \begin{equation}\label{2021-06-22+11}
\E \left(X^{r-\frac{{\hat r}}{k(k-1)}}Q^{\hat r}\right) < \infty, \qquad 
{\text{for each}}\qquad
r=2,\dots, k.
\end{equation}
where we denote ${\hat r}:=\binom{r-1}{2}+1$.
\end{rem}

The proof of Remarks
 \ref{remark-1} and \ref{remark-clique_normal}
 is presented in Section 2. Let us briefly explain the result and conditions of Theorem \ref{normal_limit_B}. Let $N_{F,i}$ be the number of copies of $F$ in $G(X_i,Q_i)$ and define 
 $S_F=N_{F,1}+\cdots+N_{F,m}$. The moment condition $\E X<\infty$ and the assumption $m\approx \nu n$ 
 control the amount of overlap between the different layers 
 $G_{n,i}$
 and ensures that (the  layer sizes)  
 ${\tilde X}_i=X_i$, $1\le i\le m$ with high probability. 
 The principal contribution to the subgraph count  ${\cal N}_F$ comes from the subgraph counts $N_{F,i}$ of individual layers (recall that $F$ is $2$-connected). Therefore we have ${\cal N}_F\approx S_F$. To make this approximation rigorous we introduce  conditions
 (\ref{2021-06-29}) and (\ref{2021-06-22+11}) 
  aimed at controlling the number of overlaps of different copies of $F$ in $G_{[n,m]}$. 
 (The combinatorial origin of  (\ref{2021-06-29}), (\ref{2021-06-22+11})
 is explained  in Lemmas  
\ref{two_connected} -- \ref{h_complete}). Finally, the asymptotic normality of ${\cal N}_F$ follows from the asymptotic normality of $S_F$. The latter is guaranteed by the second moment condition $\sigma_F^2<\infty$.

In the case where  $F$ is balanced and 
the random variable $N_F^*$ has an infinite second moment, we can obtain an $\alpha$-stable limiting distribution for the subgraph count ${\cal N}_{F}$. In Theorem \ref{stable_limit_B} below we assume that for some $a>0$ and $0<\alpha<2$, we have
\begin{equation}
\label{2021-06-23+1}
\PP\{N_F^*>t\}=(a+o(1))t^{-\alpha} \qquad {\text{as}} \quad t\to+\infty.
\end{equation}
Let $N_{F,i}^*=\E(N_F|X_i, Q_i)$, $1\le i\le m$, be iid copies of $N_F^*$ and put $S_F^*=N_{F,1}^*+\cdots+N_{F,m}^*$. It is well known (Theorem 2 in $\S$ $35$ of \cite{GnedenkoKolmogorov}) that the distribution of $m^{-1/\alpha}(S_F^*-B_m)$ converges to a stable distribution, say $G_{\alpha,a}$, which is defined by $a$ and $\alpha$. Here $B_m=m\E N_F^*=\E N_F$ for $1<\alpha <2$ and $B_m\equiv 0$ for $0<\alpha<1$. For $\alpha=1$ we have $B_m=c^{\star}_{\alpha,a}\ln m$, where the constant $c^{\star}_{\alpha,a}>0$ depends on $a$ and $\alpha$.

Our second  result establishes an $\alpha$-stable aproximation to the distribution  of 
${\cal N}_F$.
\begin{tm}\label{stable_limit_B} 
Let $\nu>0$. Let $n,m\to+\infty$ and  assume that  $m/n\to\nu$.
Let $F$ be a balanced and $2$-connected graph with $v_F\ge 3$ vertices.  Let $a>0$ and $0<\alpha<2$. 
 Assume that  $\E X<\infty$ and that (\ref{2021-06-23+1}) holds. 
 Assume, in addition, that
 \begin{equation}
\label{2021-06-29_stable}
\E \left(X^{1+s\left(1-\frac{1}{\alpha \,e_F}\right)}Q^{s}\right)<\infty
\qquad
{\text{for each}}
\qquad
s=1,\dots,v_F-1.
\end{equation}
 Then $({\cal N}_F-B_m)/m^{1/\alpha}$ converges in distribution to  $G_{\alpha,a}$.
\end{tm}

\begin{rem}\label{remark-clique_stable}
In the special case, where $F$ is a clique on $k\ge 3$ vertices ($F={\cal K}_k$) 
condition (\ref{2021-06-29_stable})
can be replaced by the following one
 \begin{equation}\label{2021-06-23+3}
 \E \left(X^{r-{\hat r}\frac{2}{\alpha k(k-1)}}Q^{\hat r}\right) < \infty
  \qquad 
{\text{for each}}\qquad
r=2,\dots, k.
\end{equation}
where  ${\hat r}=\binom{r-1}{2}+1$.
\end{rem}

 The result of Theorem \ref{stable_limit_B}  is obtained using the same approximation ${\cal N}_F\approx S_F$ as above. In addition, we use the
 observation that
condition (\ref{2021-06-23+1}) implies $S_F\approx S_F^*$.
To make this approximation rigorous we apply
 exponential large deviation bounds  \cite{JansonORucinski}
 to individual subgraph counts $N_{F,i}$ conditionally given $(X_i,Q_i)$,
 see Lemma \ref{lemma_stable}. The $\alpha$-stable limit 
  of
 $S_F^*$ is now guaranteed by condition (\ref{2021-06-23+1}) and Theorem 2 in $\S$ $35$ of \cite{GnedenkoKolmogorov}.

We briefly comment on technical  conditions  (\ref{2021-06-29}), (\ref{2021-06-22+11}),  (\ref{2021-06-29_stable}) and
(\ref{2021-06-23+3}).
The mixed moments defined 
there
appear in our  upper bounds on the expected number of
 overlaps of different copies of $F$ in $G_{[n,m]}$, see  Lemmas \ref{two_connected}, \ref{h_complete} and inequality (\ref{2021-03-18+2}) in the proof below. More precisely,
 we
 use these 
 moments to upper bound the quantity $h_F$ of (\ref{2021-03-18+2}).
 Alternatively, one can impose conditions on the rate of decay of $h_F$ directly.
 We note that for particular graphs $F$ the moment conditions
  (\ref{2021-06-29}), (\ref{2021-06-22+11}),  (\ref{2021-06-29_stable}),
(\ref{2021-06-23+3}) can be relaxed.

Let us examine  Theorems \ref{normal_limit_B}
and 
 \ref{stable_limit_B} 
in
 the special case where the marginals $X, Q$ of $(X,Q)$
are independent
 and $\PP\{Q>0\}>0$.  
 We first consider Theorem
  \ref{normal_limit_B}.
 The
 finite variance condition 
$\sigma_F^2<\infty$  
of Theorem \ref{normal_limit_B}
reduces to the moment condition 
$\E X^{2v_F}<\infty$.
Indeed, by the simple inequality $N_F\le (X)_{v_F}$, we have
that
$\E X^{2v_F}<\infty\Rightarrow \E N_F^2<\infty\Rightarrow \sigma_F^2<\infty$. On the other hand, by the variance identity
$\Var N_F=\Var N_*+\E(\Var( N_F|X,Q))$, we have that
$\sigma_F^2<\infty\Rightarrow \E (N_*)^2<\infty$, where
 the latter inequality (for independent $X$ and $Q$)
  implies $\E X^{2v_F}<\infty$.
 Moreover,  
the moment condition $\E X^{2v_F}<\infty$  implies  (\ref{2021-06-29}). Therefore 
Theorem \ref{normal_limit_B} establishes the asymptotic normality under the minimal second moment condition $\sigma_F^2<\infty$.

 We now turn to Theorem \ref{stable_limit_B}.
 For independent $X$ and $Q$
 condition  (\ref{2021-06-23+1}) of
  Theorem \ref{stable_limit_B}
  is equivalent to the condition
\begin{equation}
\label{2021-07-05}
\PP\{X>t\}=(b+o(1))t^{-\gamma} \qquad {\text{as}} \qquad t\to+\infty,
\end{equation}
where $\gamma=\alpha v_F$ and where $b$ solves the equation
$a=b(a_F/v_F!)^{\gamma/v_F}\E Q^{\gamma e_F/v_F}$.
Note that  $\E X<\infty$ implies
$\gamma>1$. Furthermore, the inequality $v_F\le e_F$ (which holds for any 2-connected $F$ with $v_F\ge 3$) combined with 
$\gamma>1$
implies 
$\alpha e_F>1$.
Observe that   for $\alpha e_F>1$
condition (\ref{2021-06-29_stable})
 reads as $\E X^{1+(v_F-1)\left(1-\frac{1}{\alpha e_F}\right)}<\infty$. In view of  (\ref{2021-07-05}) the latter expectation
 is finite  whenever
\begin{equation}\label{xxx}
 1+(v_F-1)\left(1-\frac{1}{\alpha e_F}\right)<\gamma.
\end{equation}
 We have arrived to the following corollary.

\begin{cor}\label{Corollary}
Let $\nu>0$. Let $n,m\to+\infty$ and  assume that  $m/n\to\nu$.
Let $F$ be a  $2$-connected graph with $v_F\ge 3$ vertices. 
Assume that  $X$ and $Q$ are independent and $\PP\{Q>0\}>0$.

(i) If  $\E X^{2v_F}<\infty$
then
 $({\cal N}_F-\E {\cal N}_F)/(\sigma_{F}\sqrt{m})$ converges in distribution to  the standard normal distribution.

(ii)
Let
 $b>0$ and $1<\gamma<2v_F$.
 Assume that
  (\ref{2021-07-05}) holds.
 Asumme, in addition, that 
 $F$ is balanced and 
 (\ref{xxx}) holds, where
  $\alpha=\gamma/v_F$.
  Then $({\cal N}_F-B_m)/m^{1/\alpha}$ converges in distribution to $G_{\alpha,a}$. Here $B_m$ and $G_{\alpha,a}$ are the same as in Theorem \ref{stable_limit_B}, with 
  $a=b(a_F/v_F!)^{\gamma/v_F}\E Q^{\gamma e_F/v_F}$.
\end{cor}

It is relevant to mention that the moment condition $\E X<\infty$ together with the assumption $m\approx \nu n$ imply the existence of an asymptotic degree distribution of $G_{[n,m]}$ as $n,m\to+\infty$. An asymptotic power-law degree distribution is obtained if we choose an appropriate distribution for the layer type $(X,Q)$. Furthermore, under an additional moment condition 
$\E X^3Q^2<\infty$ the random graph $G_{[n,m]}$ has a non-vanishing  global clustering coefficient, see \cite{Lasse_Mindaugas2019}. Therefore, Theorems \ref{normal_limit_B} and \ref{stable_limit_B} establish the 
limit distributions of subgraph counts in a highly clustered complex network. 
 
 Finally, we discuss an important question about the relation between the community size  $X$ and strength $Q$. In Theorems \ref{normal_limit_B} and \ref{stable_limit_B}, no assumption has been made about the stochastic dependence between the marginals $X$ and $Q$ of the bivariate random vector $(X,Q)$ defining the random graph  $G_{[n,m]}$. To simplify the model we can assume that $X$ and $Q$ are independent, see Corollary \ref{Corollary} above. However, for  network modelling purposes, various types of 
 dependence between $X$ and $Q$ are of interest. For example, a negative correlation between $X$ and $Q$ would 
emphasize small strong communities and large weak communities, a pattern likely to occur in real networks with overlapping communities. Assuming that $Q$ is proportional to a negative power of $X$, for example, $Q=\min\{1, bX^{-\beta}\}$ for some $\beta\ge 0$ and $b>0$
 (cf.\cite{Yang_Leskovec2012}, \cite{Yang_Leskovec2014}), 
 one obtains a  mathematically tractable network model admitting tunable power-law  degree and bidegree distributions  and rich clustering spectrum
 \cite{Lasse_Mindaugas2019, Joona_Lasse_Mindaugas2021}.

 {\it Related work.} 
 Asymptotic distributions of subgraph counts in Bernoulli random graphs  
 is a well  established area of research, see,
 e.g., \cite{JansonLuczakRucinski2001}, \cite{Rucinski1988} and references therein. For a recent development 
 we refer to \cite{HladkyPelekisSileikis2021},
 \cite{PrivaultSerafin2020}, 
 \cite{Rollin}.
 A signifficant difference between the sparse Bernoulli random graphs and 
 complex networks is that the former ones  have none or very few copies of   a triangle or a larger clique, while the latter ones often have abundant numbers of those.  
 The 
 abundance of dense subgraphs in otherwise sparse
 complex networks is related to the clustering property. 
 The global and local clustering coefficients
 are expressed in terms of counts of triangles and wedges.
  Therefore, a rigorous asymptotic analysis of clustering coefficients in large random networks reduces to that  of  the triangle counts and wedge counts. In particular, the bivariate asymptotic normality for triangle and wedge counts in a related sparse random intersection graph was shown in \cite{JerzyMindaugas_2018}, and related $\alpha$-stable limits were established in \cite {ValentasMindaugas_2016}. Another line of 
  research  pursued in \cite{Grohn_Karjalainen_Leskela, KarjalainenLeeuwaardenLeskela}  addresses the concentration  of subgraph counts in $G_{[n,m]}$.

\medskip

The rest of the paper is organized as follows: In 
Section 2 we formulate and prove 
Theorems \ref{normal_limit_B}, \ref{stable_limit_B} 
and  Remarks \ref{remark-1}, \ref{remark-clique_normal}, \ref{remark-clique_stable}. 
We mention that combinatorial Lemmas 
\ref{obs4}, \ref{E-H} and  
inequality  (\ref{2021-06-28+4}), see below,
may be of independent interest.

\section{Proofs}

\subsection{Notation}
Before the proof we introduce some notation.
We denote for short ${\mathbb X}=(X_1,\dots, X_m)$ and ${\mathbb Q}=(Q_1,\dots, Q_m)$. By $\E^*(\cdot)=\E(\cdot|{\mathbb X}, {\mathbb Q})$ and $\PP^*(\cdot)=\PP(\cdot|{\mathbb X}, {\mathbb Q})$ we denote the conditional expectation and probability given $({\mathbb X}, {\mathbb Q})$. 
Recall that  $a_F$ stands for
 the number of copies of $F$ in 
 the complete graph on $v_F$ vertices 
 ${\cal K}_{v_F}$. For example $a_{{\cal K}_k}=1$ and $a_{{\cal C}_k}=(k-1)!/2$. Given $F$, for any positive sequences $\{a_n\}$ and $\{b_n\}$ we denote $a_n\asymp b_n$ (respectively $a_n\prec b_n$) whenever for sufficiently large $n$ we have $c_1 \le a_n/b_n\le c_2$ (respectively $a_n\le c_2b_n$), where constants $0<c_1<c_2$ may only depend on $v_F$.

Recall that $N_F$ and $N_{F,i}$ denote the numbers of copies of $F$ in $G(X,Q)$ and $G(X_i,Q_i)$, respectively. Furthermore,
$
N_{F}^*=\E(N_F|X,Q)$,
$N_{F,i}^*=\E(N_{F,i}|X_i,Q_i)$, and 
\begin{displaymath}
S_F=N_{F,1}+\cdots+N_{F,m}, 
\qquad
S_F^*=N_{F,1}^*+\cdots+N_{F,m}^*.
\end{displaymath}
 Note that $N_{F,i}^*=\E^*(N_{F,i})$ and $S_F^*=\E^*(S_F)$. Finally, let  ${\tilde N}_i$ be the number of copies of $F$  in $G_{n,i}$ and let ${\tilde S}={\tilde N}_1+\dots+{\tilde N}_m$.

We can identify the indices $1\le i\le m$ with colors, and assign 
(the edges of) each $G_{n,i}$ the color $i$. 
The colored graph is 
denoted by $G^{\star}_{n,i}$. The union of colored graphs 
$G^{\star}_{n,1}\cup\cdots\cup G^{\star}_{n,m}$ defines a 
multigraph, denoted by $G_{[n,m]}^{\star}$, that admits parallel 
edges of different colors. 
Furthermore each edge $u\sim v$ of $G_{[n,m]}$ is 
assigned the set of colors that correspond to parallel edges of 
$G^{\star}_{[n,i]}$ connecting $u$ and $v$.

A subgraph $H\subset G_{[n,m]}$ is called monochromatic if it is a 
subgraph of some $G_{n,i}$ and none of edges of $H$ are assigned 
more than one color. Otherwise $H$ is called polychromatic. 
${\cal N}_{F,M}$ and ${\cal N}_{F,P}$ stand for the 
numbers of monochromatic and polychromatic copies of $F$ in 
$G_{[n,m]}$. A subgraph $H^{\star}\subset G_{[n,m]}^{\star}$ is 
called monochromatic if it is a subgraph of some 
$G_{n,i}^{\star}$. It is 
called polychromatic if it contains edges of different colors. Given 
$H^{\star}\subset G_{[n,m]}^{\star}$, 
let $H_0\subset G_{[n,m]}$ be the 
graph obtained from $H^{\star}$ by merging parallel edges. We call 
$H_0$ the projection of 
$H^{\star}$.
Let ${\cal N}_{F,P}^{\star}$ be the number of 
polychromatic copies of $F$  in $G_{[n,m]}^{\star}$.
Note that there can be 
several monochromatic and/or polychromatic copies of $F$ in
 $G_{[n,m]}^{\star}$  sharing the same projection $F_0$.
We fix a copy $F_0\subset G_{[n,m]}$ of $F$ and denote 
by $h_F$  the expected number of polychromatic subgraphs
 of $G^{\star}_{[n,m]}$ whose projection is $F_0$. Clearly, the number $h_F$ does not depend on the location of  $F_0$.
  
\medskip

\subsection {Proofs} 
We start with an outline of the proof.
We approximate ${\cal N}_F\approx {\tilde S}_F$ and ${\tilde S}_F\approx S_F$. In the case where $\E N_{F}^2<\infty$ we deduce the normal approximation to the sum $S_F$ (of iid random variables) by the standard central limit theorem. In the case where $N_F$ has an 
infinite variance we further approximate $S_F\approx S_F^*$ and deduce the $\alpha$-stable approximation by the generalized central limit theorem 
(see Theorem 2 in $\S$ $35$   
 of \cite{GnedenkoKolmogorov}).

 {\it Approximation  ${\cal N}_F\approx {\tilde S}_F$.} The approximation  follows from  the simple observation that 
 \begin{equation}
\label{2021-05-18}
{\cal N}_F= {\cal N}_{F,M}+{\cal N}_{F,P}, \qquad 
{\cal N}_{F,M}\le {\tilde S}_F\le {\cal N}_{F,M} + {\cal N}^{\star}_{F,P}, \qquad 
{\cal N}_{F,P}\le {\cal N}_{F,P}^{\star}.
\end{equation}
 We only comment on  the second inequality. To see why it holds true\JKcorr{,} let us inspect every copy of  $F$ in $G_{[n,m]}$ that belongs to two or more layers $G_{n,i}$. Let $F_0\subset G_{[n,m]}$ be such a copy. Clearly, the number of polychromatic subgraphs $F^{\star}$ in $G^{\star}_{[n,m]}$, whose projection $F^{\star}_0$ is $F_0$, is larger than the number of monochromatic ones. Hence ${\tilde S}_F\le {\cal N}_{F,M}+{\cal N}^{\star}_{F,P}$. From (\ref{2021-05-18}) we conclude that 
\begin{equation}\label{2021-04-09+13}
|{\tilde S}_F-{\cal N}_F|\le {\cal N}_{F,P}^{\star}.
\end{equation}

\smallskip

In order to assess the accuracy of the approximation 
${\cal N}_F\approx {\tilde S}_F$, we evaluate the expected value of ${\cal N}_{F,P}^{\star}$. Let ${\cal K}_{[n]}$ be a clique on the vertex set $V=[n]$. We couple $G_{[n,m]}\subset {\cal K}_{[n]}$ and fix a subgraph $F_0=({\cal V}_{0},{\cal E}_{0})\subset {\cal K}_{[n]}$ with vertex set $\{1,\dots, v_F\}\subset V$, which is a copy of $F$.
We have, by symmetry,
\begin{equation}\label{2021-03-18+2}
\E {\cal N}_{F,P}^{\star}=\binom{n}{v_F}a_Fh_F.
\end{equation}
Each $F^{\star}\subset G^{\star}_{[n,m]}$ whose projection  is $F_0$
 is defined by the partition of the edge set ${\cal E}_{0}$ into non-empty color classes, say\JKcorr{,} $B_1\cup\cdots\cup B_r={\cal E}_0$, and the vector of distinct colors $(i_1,\dots, i_r)\subset[m]^r$ such that all the edges in $B_j$ are of the color $i_j$ (edges of $B_j$ belong to $G^{\star}_{n,i_j}$). Denote by ${\tilde B}=(B_1,\dots, B_r)$ and ${\tilde i}=(i_1,\dots, i_r)$ the partition and its coloring. The polychromatic subgraph $F^{\star}$  defined by the pair $({\tilde B}, {\tilde i})$ is denoted  $F({\tilde B}, {\tilde i})$. The probability that such a subgraph is present in $G_{[n,m]}^{\star}$ is 
\begin{equation}
\label{2021-06-28+5}
h({\tilde B}, {\tilde i}) := \PP\bigl\{ F({\tilde B}, {\tilde i})\in G_{[n,m]}^* \bigr\} = \prod_{j=1}^r \frac{1}{(n)_{v_j}} \E \left( ({\tilde X}_{i_j})_{v_j}Q_{i_j}^{b_j} \right).
\end{equation}
Here $b_j:=|B_j|$, and $v_j$ is the number of distinct vertices incident to edges from $B_j$. We have
\begin{equation}\label{2021-03-01}
h_F = \E \left( \sum_{({\tilde B}, {\tilde i})} {\mathbb I}_{ \bigl\{ F({\tilde B}, {\tilde i})\in G_{[n,m]}^* \bigr\} } \right) = \sum_{({\tilde B}, {\tilde i})} h({\tilde B}, {\tilde i}).
\end{equation}
Here the sum runs over all  possible polychromatic $F^{\star}$ whose projection $F^{\star}_0$ is $F_0$. We
upper
 bound $h_F$ in Lemmas \ref{two_connected} and \ref{h_complete} below.

\medskip

\noindent

{\it Approximation ${\tilde S}_F\approx S_F$.} For $1\le i\le m$ we couple $G({\tilde X}_i,Q_i)\subset G(X_i,Q_i)$ and ${\tilde N}_i\le N_i$ so that $G({\tilde X}_i,Q_i)\not= G(X_i,Q_i)$ and ${\tilde N}_i\not= N_i$ whenever $X_i>n$. For $m=O(n)$\JKcorr{,} the event ${\cal A}_n:=\{\max_{1\le i\le m}X_i>n\}$ has probability 
\begin{equation}
\label{2021-06-28+1}
\PP\{{\cal A}_n\} \le \sum_{i=1}^m\PP\{X_i>n\} 
\le 
\frac{m}{n}
\E \left(X_1{\mathbb I}_{\{X_1>n\}}\right) = o(1)\JKcorr{,}
\end{equation}
hence $\PP\{{\tilde S}_F\not=S_F\}=o(1)$.
In 
(\ref{2021-06-28+1}) 
we used the fact that $\E X_1<\infty
\Rightarrow
\E \bigl(X_1{\mathbb I}_{\{X_1>n\}}\bigr)=o(1)$.
\smallskip

\begin{proof}[Proof of Theorem  \ref{normal_limit_B} 
and Remark
\ref{remark-clique_normal}]
By Lemma \ref{two_connected} (respectively, Lemma \ref{h_complete}),
we have 
 $h_F=o(n^{0.5-v_F})$.
Invoking this bound in  (\ref{2021-03-18+2}) we obtain 
 ${\cal N}_{F,P}^{\star}=o_P(\sqrt{m})$. 
 Next, from   (\ref{2021-04-09+13}) we obtain that
 $({\cal N}_F-S_F)=o_P(\sqrt{m})$.
  Then an application of  (\ref{2021-06-28+1})
shows $({\cal N}_F-S_F)=o_P(\sqrt{m})$. Finally, we apply the classical central limit theorem to the sum of iid random variables $S_F$ to get the asymptotic normality of $({\cal N}_F-\E {\cal N}_F)/ (\sigma_{F}\sqrt{m})$.
\end{proof}

\begin{proof}[Proof of Theorem  \ref{stable_limit_B}
and Remark
\ref{remark-clique_stable}]
By Lemma \ref{two_connected} (respectively, Lemma \ref{h_complete}),
we have 
 $h_F=o(n^{\frac{1}{\alpha}-v_F})$.
Proceeding as in the proof of Theorem  \ref{normal_limit_B} above\JKcorr{,} we obtain $({\cal N}_F-S_F)=o_P(m^{1/\alpha})$. Furthermore, by Lemma \ref{lemma_stable}, the iid random variables $N_{F,1}$, $N_{F,2}$, $\dots$ obey  the power law (\ref{2021-06-11+1}) and therefore $(S_F-B_m)/m^{1/\alpha}$ converges in distribution to $G_{\alpha,a}$ (see Theorem 2 in $\S$ $35$ of \cite{GnedenkoKolmogorov}). Hence $m^{-1/\alpha}({\cal N}_F-B_m)$ converges in distribution to  $G_{\alpha,a}$.
\end{proof}

\begin{proof}[Proof of Remark \ref{remark-1}]
 We have $\sigma_F^2=\Var N_F=\Var N_F^*+\E (\Delta_F^*)^2$, where $\Delta_F^*:=N_F-N_F^*$. Therefore $\sigma_F^2<\infty\Rightarrow \Var N_F^*<\infty \Rightarrow \E (N_F^*)^2<\infty$. To prove that $\E (N_F^*)^2<\infty\Rightarrow \sigma_F^2<\infty$, it suffices to show that $\E (\Delta_F^*)^2<\infty$. By Lemma 3.5 of \cite{JansonLuczakRucinski2001}, we have $\E^*(\Delta_F^*)^2\prec (N_F^*)^2/\Phi_F(X,Q)$, where $\Phi_F(X,Q)=\min_{H\subset F}X^{v_H}Q^{e_H}$.
 Furthermore, from the inequality (\ref{2021-06-21}), 
see below,
 which holds for balanced $F$, we obtain
\begin{displaymath}
 \E^*(\Delta_F^*)^2 \prec\frac{ (N_F^*)^2}{\min\{(N_F^*)^{2/v_F},N_F^*\}} = 
 \max \{ (N_F^*)^{2-2/v_F}, N_F^*\} \le \max \{1, (N_F^*)^2\}.
\end{displaymath}
Hence $\E (N_F^*)^2<\infty$ implies $\E (\Delta_F^*)^2=\E (\E^*(\Delta_F^*)^2)<\infty$.
\end{proof}

\subsection{Auxiliary lemmas}

In Lemmas \ref{two_connected} and \ref{h_complete} we 
 upper bound the moments $h_F$ for 
 $2$-connected $F$ and for $F={\cal K}_k$, respectively. Clearly, the 
result of Lemma \ref{two_connected} applies to  $F={\cal K}_k$ as 
well, but the bound of Lemma  \ref{h_complete} is tighter for large 
$k$.

\begin{lem}\label{two_connected}
Let $F$ be a $2$-connected graph with $v_F\ge 3$ vertices. Let $n,m\to+\infty$. Assume that $m=O(n)$.
 
(i)  Assume that (\ref{2021-06-29}) holds. Then $h_F=o(n^{0.5-v_F})$.

(ii) Assume that $0<\alpha<2$, and that  
(\ref{2021-06-29_stable}) holds. Then $h_F=o(n^{\frac{1}{\alpha}-v_F})$.
\end{lem}

In the proof we use the simple fact that for any $s,t,\tau>0$, the moment condition $\E (X^sQ^t)<\infty$ implies
\begin{equation}\label{fact_1}
\E \left((\min\{X,n\})^{s+\tau}Q^t\right) = o\left(n^{\tau}\right).
\end{equation}
Denote ${\tilde X}:=\min\{X,n\}$. To see why (\ref{fact_1}) 
holds, choose $0<\delta<\tau/(s+\tau)$ and split the expectation
\begin{displaymath}
\E \left({\tilde X}^{s+\tau}Q^t\right) = 
\E \left( {\tilde X}^{s+\tau}Q^t {\mathbb I}_{\{X<n^{\delta}\}} \right) +
\E \left( {\tilde X}^{s+\tau}Q^t{\mathbb I}_{\{X\ge n^{\delta}\}} \right) =: I_1+I_2.
\end{displaymath}
Inequalities ${\tilde X}\le n$ and $\E (X^sQ^t)<\infty$ imply 
$I_2
\le 
n^{\tau} \E \left(X^sQ^t{\mathbb I}_{\{X\ge n^{\delta}\}}\right) 
=
n^{\tau}\cdot o(1)$.
Inequality ${\tilde X}\le X$ implies $I_1\le n^{\delta(s+\tau)}=o(n^\tau)$.

\begin{proof}[Proof of Lemma \ref{two_connected}]
The proofs of statements (i) and (ii) are identical. Therefore we only 
prove statement (i).  

We start with establishing  an auxiliary inequality 
(\ref{2021-06-28+4}) below, which may be interesting in itself. 
Let $r\ge 2$. Given a partition  ${\tilde B}=(B_1,\dots, B_r)$ of the 
edge set ${\cal E}_0$ of the graph $F_0=({\cal V}_0,{\cal E}_0)$, 
and given
$i\in [r]$, let $V_i$ be the set of vertices  incident to the edges from 
$B_i$. Let $\rho_i$ be the number of (connected) components of 
the graph $Z_i=(V_i,B_i)$ and put $v_i=|V_i|$. We claim that 
\begin{equation}
\label{2021-06-28+4}
v_1+\dots+v_r\ge v_F+\rho_1+\dots+\rho_r.
\end{equation}
To establish the claim we consider the list $H_1,H_2,\dots, H_t$ of components of $Z_1$, $\dots$, $Z_r$ arranged in an arbitrary order. Here $t:=\rho_1+\dots+\rho_r$. Therefore, each graph $H_i$ is a component of some $Z_j$ and their union $H_1\cup\cdots\cup H_t=Z_1\cup\cdots\cup Z_r=F_0$. Let us consider the sequence of graphs ${\bar H}_j:= H_1\cup\cdots\cup H_{j}$, for $j=1,\dots, t-1$. Let ${\bar \rho}_j$ and ${\bar v}_j$ denote the number of components  and the number of vertices of  ${\bar H}_j$. Let $v'_j$ denote the number of vertices of $H_j$. We use the observation that 
\begin{equation}\label{2021-06-28+6}
{\bar v}_j \le {\bar v}_{j-1}+v'_j+{\bar \rho}_{j}-{\bar \rho}_{j-1}-1 \qquad
{\text{for}} \qquad j=2,\dots t-1.
\end{equation}
Indeed, ${\bar \rho}_{j-1}={\bar \rho}_j$ means that the vertex set 
of (the connected graph) $H_j$ intersects with exactly one 
component of ${\bar H}_{j-1}$. Consequently, 
$H_j$ and ${\bar H}_{j-1}$ have at least one common vertex and 
therefore  (\ref{2021-06-28+6}) holds. Similarly, 
${\bar \rho}_{j-1}-{\bar \rho}_j=y>0$ 
means that the vertex set of $H_j$ intersects with exactly $y+1$ 
different components of ${\bar H}_{j-1}$. Consequently, $H_j$ and 
${\bar H}_{j-1}$ have at least $y+1$ common vertices and  
(\ref{2021-06-28+6}) holds again. The remaining case 
${\bar \rho}_{j-1}-{\bar \rho}_j=-1$ is realized by the configuration 
where the vertex sets of $H_j$ and ${\bar H}_{j-1}$ have no 
common elements. In this case (\ref{2021-06-28+6}) follows from 
the identity ${\bar v}_j = {\bar v}_{j-1}+v'_j$.

By summing up the inequalities (\ref{2021-06-28+6}), we obtain (using ${\bar \rho}_1=1$) that
\begin{displaymath}
{\bar v}_{t-1} \le  v'_1+\cdots+v'_{t-1}+ {\bar \rho}_{t-1}-t+1.
\end{displaymath}
Note that given ${\bar H}_{t-1}$ with ${\bar \rho}_{t-1}$ components, the vertex set of  $H_{t}$  must intersect with each component in two or more points, in order to make the union 
${\bar H}_{t-1}\cup H_t=F_0$ $2$-connected. Consequently, we have
\begin{displaymath}
{\bar v}_t \le {\bar v}_{t-1}+v'_t-2{\bar \rho}_{t-1}.
\end{displaymath}
Finally, we obtain
\begin{displaymath}
v_F = {\bar v}_t \le v_1'+\cdots+v'_t-{\bar \rho}_{t-1}-t+1.
\end{displaymath}
Now the claim follows from the identity $v'_1+\dots+v'_t=v_1+\cdots+v_r$ and the inequality ${\bar \rho}_{t-1}\ge 1$.

\medskip

\noindent
Let us prove statement (i). Given $({\tilde B}, {\tilde i})$, we obtain from  (\ref{2021-06-28+5})   and (\ref{2021-06-28+4}) 
(recall  the notation $b_j=|B_j|$) that 
\begin{align*}
 h({\tilde B}, {\tilde i}) \le \frac{1}{n^{v_1+\dots+v_r}}\prod_{j=1}^r \E \left(
{\tilde X}^{v_j}Q^{b_j} \right)
\le
\frac{1}{n^{v_F+\rho_1+\dots+\rho_r}} \prod_{j=1}^r \E\left( {\tilde X}^{v_j}Q^{b_j}\right).
\end{align*}
Given ${\tilde B}=(B_1,\dots, B_r)$, we estimate the sum over all possible colorings (there are $(m)_r$ of them) 
\begin{align*}
\sum_{{\tilde i}}h({\tilde B},{\tilde i}) & \prec \frac{(m)_r}{n^{v_F+\rho_1+\dots+\rho_r}} \prod_{j=1}^r \E\left( {\tilde X}^{v_j}Q^{b_j} \right)
\asymp
n^{-v_F} \prod_{j=1}^r \frac{\E\left( {\tilde X}^{v_j}Q^{b_j}\right)}{n^{\rho_j-1}} \\
&= n^{0.5-v_F} \prod_{j=1}^r \frac{\E\left( {\tilde X}^{v_j}Q^{b_j}\right)}{n^{\rho_j-1+(b_j/(2e_F))}} = o\bigl(n^{0.5-v_F}\bigr).
\end{align*}
In the second last identity we used $b_1+\dots+b_r=e_F$, while the last bound follows by the chain of inequalities 
\begin{align*}
n^{1-\rho_j}\E\left({\tilde X}^{v_j}Q^{b_j}\right) &\le
\E\left({\tilde X}^{v_j+1-\rho_j}Q^{b_j}\right) \le
\E\left({\tilde X}^{v_j+1-\rho_j}Q^{v_j-\rho_j}\right) \\
&= o\left(n^{(v_j-\rho_j)/(2e_F)} \right) = o\left( n^{b_j/(2e_F)} \right).
\end{align*}
Here in the first step we used ${\tilde X}\le n$; in the second step we used $Q\le 1$ and 
$b_j\ge v_j-\rho_j$ (the latter inequality is based on the observation that any graph with $v_j$ vertices and $\rho_j$ components  has at least $v_j-\rho_j$ edges); the third step follows by (\ref{fact_1}) from the moment condition (\ref{2021-06-29}) applied to $s=v_j-\rho_j$; the last step follows from the inequality $b_j\ge v_j-\rho_j$.

Finally, we conclude that 
\begin{equation}\label{2021-06-29+10}
h_F = \sum_{{\tilde B}} \sum_{{\tilde i}} h({\tilde B},{\tilde i}) = o\bigl(n^{0.5-v_F}\bigr),
\end{equation}
because the number of  partitions ${\tilde B}$ of the edge set of a given graph $F$ is always finite. 
\end{proof}

Before showing an upper bound for $h_F$, $F={\cal K}_k$, we introduce some notation. Given an integer $b\ge 1$, let $b^{\star}$ be the minimal number of vertices that a graph with $b$ edges may have. Let $H_b$ be such a graph. It has a simple structure described below.
Let $k_b\ge 2$ be the largest integer satisfying $b\ge \binom{k_b}{2}$. Then 
\begin{displaymath}
b={\binom{k_b}{2}} +\Delta_b
\end{displaymath}
for some integer $0\le \Delta_b\le k_b-1$. For  $\Delta_b=0$ we have $b^{\star}=k_b$ and $H_b={\cal K}_{b^{\star}}$ (clique on $b^{\star}=k_b$ vertices). For $\Delta_b>0$, graph $H_b$ is a union of  ${\cal K}_{k_b}$ and a star ${\cal K}_{1, \Delta_b}$, such that all the vertices of the star except for the central vertex belong to the vertex set of ${\cal K}_{k_b}$. In this case $b^{\star}=k_b+1$. In other words, one obtains $H_b$  from ${\cal K}_{k_b+1}$ by deleting $k_b-\Delta_b$ edges sharing a common endpoint. The next two lemmas establish useful properties of the function $b\to b^{\star}$.

\begin{lem}\label{obs4} For integers $s\ge t\ge 1$ we have 
\begin{equation}\label{2021-03-16+10}
s^{\star}+t^{\star}\ge (s+t-1)^{\star}+2. 
\end{equation}
\end{lem}

\begin{proof} 
In the proof we consider graphs $H_s$ and $H_t$ that have disjoint vertex sets so that the union $H_s\cup H_t$ has $s^{\star}+t^{\star}$ vertices.

Note that for $t=1$ both sides of (\ref{2021-03-16+10}) are equal. In order to show (\ref{2021-03-16+10}) for $s\ge t\ge 2$ we consider the chain of neighboring pairs
\begin{equation}\label{2021-03-16+11}
(s,t)\to (s+1, t-1)\to\cdots\to (s+t-1,1).
\end{equation} 
In a step $(x,y)\to (x+1,y-1)$  we remove an edge from $H_y$ and add it to $H_x$. A simple 
analysis of the step  $(H_x,H_y)\to (H_{x+1}, H_{y-1})$ shows that
\begin{eqnarray}\label{2021-03-16+1}
&&
 (x+1)^{\star}+(y-1)^{\star}= x^{\star}+y^{\star}+1 \ \quad
 {\text{whenever}} \ \quad
 \Delta_x=0, \ \Delta_y\not=1, \\
\label{2021-03-16+2}
&&
 (x+1)^{\star}+(y-1)^{\star}= x^{\star}+y^{\star}-1 \ \quad
  {\text{whenever}} \ \quad
\Delta_x\not= 0, \  \Delta_y=1,\\
\label{2021-03-16+3}
&&
 (x+1)^{\star}+(y-1)^{\star}= x^{\star}+y^{\star}
\quad \quad \quad {\text{in the remaining cases}}.
 \end{eqnarray} 
We call a step $(x,y)\to(x+1,y-1)$ positive (respectively negative or neutral) if (\ref{2021-03-16+2}) (respectively (\ref{2021-03-16+1}) or (\ref{2021-03-16+3})) holds. Therefore, as we move in (\ref{2021-03-16+11}) from left to right, every positive (negative) step decreases (increases) the total number of vertices in the union $H_x\cup H_y$.

Let us now traverse (\ref{2021-03-16+11}) from right to left. We 
observe that the first non-neutral step encountered 
is positive (if we encounter a non-neutral step at all). Furthermore, 
after a negative step the first non-neutral step encountered is 
positive. Note that it may happen that the last encountered non-
neutral step is negative. Therefore, the total number of positive 
steps is at least as large as the number of negative ones. This 
proves (\ref{2021-03-16+10}).
\end{proof}

\begin{lem}\label{E-H}  Let $k\ge 3$ and $r\ge 2$. Let $B_1\cup\cdots\cup B_r$ be a partition of the edge set of the clique ${\cal K}_k$. Denote $b_i=|B_i|$, $1\le i\le r$, and $\varkappa=\binom{k}{2}$. We have
\begin{equation}\label{2021-03-18}
b_1^{\star}+\cdots+b_r^{\star} \ge (\varkappa-(r-1))^{\star}+2(r-1) \ge k+r.
\end{equation}
\end{lem}

\begin{proof}
The first inequality of (\ref{2021-03-18})  follows from (\ref{2021-03-16+10}) and the identity $b_1+\dots+b_r=\varkappa$. The second inequality is simple. Indeed, for $r\ge k$ the inequality follows from  $2(r-1)\ge k+r-2$ and $(\varkappa-(r-1))^{\star}\ge 2$. For $r\le k-1$ we have $\varkappa-(r-1)\ge {\binom{k-1}{2}}+1$ and therefore $(\varkappa-(r-1))^{\star}\ge k$.
\end{proof}

Now we are ready to bound $h_F$ for $F={\cal K}_k$.

\begin{lem}\label{h_complete} 
Let $k\ge 3$, $0<\alpha\le 2$, and $A>0$. Let $n,m\to+\infty$. 
Assume that $m\le An$. Let $F={\cal K}_k$. Then 
(\ref{2021-06-23+3}) implies the bound 
$h_F=o\left(n^{\frac{1}{\alpha}-k}\right)$. 
Note that for $\alpha=2$ condition  (\ref{2021-06-23+3}) is the 
same as (\ref{2021-06-22+11}).
\end{lem}

\begin{proof}
For $F={\cal K}_k$ we have $e_F=\binom{k}{2}$.
We observe that  (\ref{2021-06-23+3}) implies
\begin{equation}
\label{2021-06-30+11}
\E\left(X^{b^{\star}-b/(\alpha\, e_F)}Q^b\right) < \infty \qquad
{\text{for each}}
\qquad
 1\le b<\binom{k}{2}.
\end{equation}
Note that ${\hat s}=\binom{s-1}{2}+1$ is the smallest integer $t$ such that  $t^{\star}=s$. In particular, for any $b$ with $b^{\star}=s$ we have $b\ge {\hat s}$. Therefore, given $2\le s\le k$, the moment condition $\E \left( X^{s-{\hat s}/(\alpha\, e_F)}Q^{\hat s} \right) < \infty$ implies $\E \left(X^{s-b/(\alpha\, e_F)}Q^{b}\right) < \infty$ for any $b$ satisfying $b^{\star}=s$. In this way (\ref{2021-06-23+3}) yields (\ref{2021-06-30+11})

Let us bound $h_{{\cal K}_k}$ from above. 
Given a partition  ${\tilde B}=(B_1,\dots, B_r)$ of the edge set ${\cal E}_0$ of ${\cal K}_k=([k],{\cal E}_0)$, let $v_j$ be the number of vertices incident to the edges from $B_j$ and let $b_j=|B_j|$. For any vector ${\tilde  i}=(i_1,\dots, i_r)$ of distinct colors we have
\begin{displaymath}
h({\tilde B},{\tilde i}) \le \prod_{j=1}^r \frac{\E\bigl({\tilde X}^{v_j}Q^{b_j}\bigr)}{n^{v_j}}
\le
\prod_{j=1}^r \frac{\E\bigl({\tilde X}^{b_j^{\star}}Q^{b_j}\bigr)} {n^{b_j^{\star}}}
\le 
\frac{1}{n^{k+r}} \prod_{j=1}^r \E\bigl({\tilde X}^{b_j^{\star}}Q^{b_j}\bigr).
\end{displaymath}
Here the first inequality follows from $({\tilde X})_t/(n)_t\le {\tilde X}^t/n^t$, since ${\tilde X}\le n$. The second inequality follows from the obvious inequality  $b_j^{\star}\le v_j$ and the fact that ${\tilde X}\le n$. The last inequality  follows from the inequality $b_1^{\star}+\cdots+b_r^{\star}\ge k+r$ of Lemma \ref{E-H}.

For each $r$-partition ${\tilde B}$ as above we bound the sum over all possible colorings ${\tilde i}$ (there are $(m)_r$ of them) 
\begin{equation}
\label{2021-06-30+10}
\sum_{{\tilde i}}h({\tilde B},{\tilde i}) \le 
\frac{(m)_r}{n^{k+r}} \prod_{j=1}^r \E\bigl({\tilde X}^{b_j^{\star}}Q^{b_j}\bigr) \le 
\frac{A^r}{n^k} \prod_{j=1}^r \E\bigl({\tilde X}^{b_j^{\star}}Q^{b_j}\bigr) =
o\left(n^{\frac{1}{\alpha}-k}\right).
\end{equation}
In the very last step we used the bounds 
(with $e_F=b_1+\cdots+b_r=\binom{k}{2}$)
\begin{displaymath}
\E\left({\tilde X}^{b_j^{\star}}Q^{b_j}\right) = o\left(n^{\frac{b_j}{\alpha\, e_F}} \right)
\end{displaymath}
that follow from the moment conditions $\E\left(X^{b_j^{\star}-\frac{b_j}{\alpha\, e_F}}Q^{b_j}\right) < \infty$, see (\ref{2021-06-30+11}), via (\ref{fact_1}). Finally, proceeding as in (\ref{2021-06-29+10}) above, we obtain the desired bound $h_F=o\left( n^{\frac{1}{\alpha}-k}\right)$ from (\ref{2021-06-30+10}).
\end{proof}

\subsection{Power-law tails}

Recall that given a graph $F=({\cal V}_F,{\cal E}_F)$, we denote 
by $v_F=|{\cal V}_F|$ the number of vertices and 
by $e_F=|{\cal E}_F|$ the number of edges.
 Let $\Psi_F=\Psi_F(n,p)=n^{v_F}p^{e_F}$, and define
\begin{displaymath}
\Phi_F=\Phi_F(n,p) = \min_{H\subset F,\, e_H\ge 1}\Psi_H, \qquad m_F = \max_{H\subset F,\, e_H\ge 1}(e_H/v_H).
\end{displaymath}
Here the minimum/maximum is taken over all subgraphs 
$H\subset F$ with $e_H\ge 1$.
Recall that $F$ is called balanced if $m_F=e_F/v_F$. 
For a balanced  $F$ we have for any $H\subset F$ with 
$e_H\ge 1$ that
\begin{displaymath}
\Psi_H = \bigl(np^{e_H/v_H}\bigr)^{v_H} \ge 
\bigl(np^{e_F/v_F}\bigr)^{v_H} = \Psi_F^{v_H/v_F}.
\end{displaymath}
Hence 
\begin{equation}\label{2021-06-21} 
\Phi_F \ge\min\{\Psi_F^{2/v_F},\Psi_F\}.
\end{equation}

\begin{lem}
\label{lemma_stable} 
Let $k\ge 3$ be an integer. Let $a>0$ and $0<\alpha<2$. Assume that $F$ is  balanced and connected. 
Assume that
\begin{equation}\label{2021-06-11}
\PP\{ N_F^*>t)=(a+o(1))t^{-\alpha} \qquad {\text{as}} \qquad t\to+\infty. 
\end{equation}
Then 
\begin{equation}
\label{2021-06-11+1}
\PP\{N_F>t\}=(a+o(1))t^{-\alpha} \qquad {\text{ as}} \qquad  t\to+\infty. 
 \end{equation}
\end{lem}

We remark that for $0<\alpha<2$, the tail asymptotics
(\ref{2021-06-11+1}) implies that $N_F$ belongs to the domain of 
attraction of an $\alpha$-stable distribution. Indeed, the left tail of  
$N_F$ vanishes since  $\PP\{N_F\ge 0\}=1$. Therefore, the 
conditions of Theorem 2 in $\S$ $35$  Chapter 7 of 
\cite{GnedenkoKolmogorov} are satisfied. 

\begin{proof}
We denote the conditional expectation and probability given 
$(X,Q)$ by $\E^*$ and $\PP^*$.  Furthermore, we 
denote $k=v_F$ and
$\Delta_F^* = N_F-N_F^*$. In the proof we often use the fact 
(see Lemma 3.5 of \cite{JansonLuczakRucinski2001}) that
\begin{equation}\label{2021-06-21+3}
\E^*(\Delta_F^*)^2\asymp \frac{(N_F^*)^2}{\Phi_F(X,Q)} (1-Q).
\end{equation}
We also use the simple relation 
$N_F^*\asymp a_F\Psi_F(X,Q)$.

To prove (\ref{2021-06-11+1}) we show that the contribution of $\Delta_F^*$  to the sum $N_F=N_F^*+\Delta_F^*$ is negligible compared to $N_F^*$ and, therefore, the tail asymptotic (\ref{2021-06-11+1}) is determined by (\ref{2021-06-11}). For this purpose we apply exponential large deviation bounds for subgraph counts in Bernoulli random graphs \cite{JansonLuczakRucinski2001, JansonORucinski}.

Given large $t>0$ and small $\varepsilon>0$, introduce event ${\cal H} = \{-\varepsilon N_F^*\le \Delta_F^*\le \varepsilon t\}$ and split
\begin{eqnarray}\nonumber
\PP\{N_F>t\} &=& \PP\{N_F>t, {\cal H}\} + \PP\{N_F>t, \Delta_F^*<-\varepsilon N_F^*\} +
\PP\{N_F>t, \Delta_F^*>\varepsilon t\} \\
\label{2021-04-02+2}
&=:& P_1+P_2+P_3.
\end{eqnarray}
We first consider $P_1$. Replacing $\Delta_F^*$ by its extreme values (on  ${\cal H}$) yields the inequalities
\begin{equation}\label{2021-04-02+1}
\PP\{(1-\varepsilon)N_F^*>t, {\cal H}\} \le P_1 \le \PP\{N_F^*>t(1-\varepsilon), {\cal H}\}.
\end{equation}
We note that the right side 
of (\ref{2021-04-02+1}) is at most $\PP\{N_F^*>t(1-\varepsilon)\}$ and the left side is at least
\begin{displaymath}
\PP\{(1-\varepsilon)N_F^*>t\} - P_2'-P_3',
\end{displaymath}
where 
\begin{displaymath}
P_2' := \PP \{ (1-\varepsilon)N_F^*>t, \Delta_F^*<-\varepsilon N_F^*\}, \qquad
P_3' := \PP\{(1-\varepsilon)N_F^*>t, \Delta_F^*>\varepsilon t\}.
\end{displaymath}
Hence, we have 
\begin{equation}
\label{2022-02-15}
\PP\{(1-\varepsilon)N_F^*>t\} - P_2'-P_3' \le P_1 \le \PP\{N_F^*>t(1-\varepsilon)\}.
\end{equation}
Invoking the simple inequalities $P_2\le P_2'$ and $P_3'\le P_3$, we obtain from  (\ref{2021-04-02+2}), {\color{blue}(\ref{2022-02-15})}
  that
\begin{equation}\label{2021-06-11+10}
 \PP\{(1-\varepsilon)N_F^*>t\} - P_2'  \le \PP\{N_F>t\} \le \PP\{N_F^*>t(1-\varepsilon)\} +P'_2+P_3.
\end{equation}
 We show below that for any $0<\varepsilon<1$ 
\begin{equation}\label{2021-04-02+3}
 P_2'=o(t^{-\alpha}) \qquad \text{and} \qquad P_3=o(t^{-\alpha}) \qquad {\text{as}}  \qquad t\to+\infty.
 \end{equation}
 Note that (\ref{2021-06-11}) and (\ref{2021-06-11+10}) together with (\ref{2021-04-02+3}) imply (\ref{2021-06-11+1}). It remains to show (\ref{2021-04-02+3}).
 
\medskip

 {\it Proof of} $ P_2'=o(t^{-\alpha})$. 
 Given  $(X,Q)$ with $0<Q<1$ (cases 0 and 1 are trivial), 
 we apply Janson's  inequality (Theorem 2.14 of 
 \cite{JansonLuczakRucinski2001}) 
 to 
 $p^*_{\varepsilon} :=\PP^*\{\Delta_F^*<-\varepsilon N_F^*\}$. 
 In what follows, we assume that the random graph  $G(X,Q)$ and 
 complete graph ${\cal K}_{X}$ are both defined on the same vertex set of size $X$ and that $X\ge 1$. Let
\begin{displaymath}
{\overline \delta} := \E^*\bigl(N_F^2\bigr) - \delta, \qquad 
\delta := \sum_{F'\subset {\cal K}_{X}} \sum_{
\begin{subarray}{c}
 F''\subset {\cal K}_{X} \\
{\cal E}_{F'}\cap {\cal E}_{F''}=\emptyset
\end{subarray}
}
\E^* ({\mathbb I}_{F'}{\mathbb I}_{F''}).
\end{displaymath}
Here the sum runs over ordered pairs $(F',F'')$ of subgraphs of 
${\cal K}_{X}$  such that $F'$ and $F''$ are copies of $F$ and their 
edge sets ${\cal E}_{F'}$ and ${\cal E}_{F''}$ are disjoint. 
Furthermore, ${\mathbb I}_{F'}$ stands for the indicator of the 
event that $F'$ is present in $G(X,Q)$. Janson's inequality implies 
\begin{equation}
\label{Janson}
 \PP^*\{\Delta_F^*<-\eta N_F^*\} \le e^{-(\eta N_F^*)^2/{\bar \delta}}, \qquad \forall \ \eta\in(0,1).
\end{equation}

Next we bound ${\bar \delta}$ from above. The (variance) identity  $\E^*(N_F^2)-(N_F^*)^2=\E^*(\Delta_F^*)^2$ implies 
\begin{equation}\label{2021-06-14}
{\overline \delta} = \E^*(\Delta_F^*)^2+(N_F^*)^2 - \delta.
\end{equation}
Furthermore, using the observation 
that $V_{F'}\cap V_{F''}=\emptyset$ implies 
${\cal E}_{F'}\cap {\cal E}_{F''}=\emptyset$, and that the latter 
relation implies 
$\E^* ({\mathbb I}_{F'}{\mathbb I}_{F''}) = (\E^* {\mathbb I}_{F'})(\E^*{\mathbb I}_{F''})=Q^{2e_F}$ we bound $\delta$ from below:
\begin{align*}
\delta \ge \sum_{F'\subset {\cal K}_{X}} \sum_{
\begin{subarray}{c}
 F''\subset {\cal K}_{X} \\
V_{F'}\cap V_{F''}=\emptyset
\end{subarray}
}
\E^* ({\mathbb I}_{F'}{\mathbb I}_{F''}) 
=
 a_F^2\binom{X}{k}\binom{X-k}{k}Q^{2e_F}
 =
\frac{(X-k)_k}{(X)_k}
 (N_F^*)^2.
\end{align*}
Then we lower bound the fraction
\begin{displaymath}
\frac{(X-k)_k}{(X)_k} 
\ge 
\left(1-\frac{k}{X-k}\right)^k 
\ge 1-\frac{k^2}{X-k}, \qquad {\text{for}} \quad X\ge 2k,
\end{displaymath}
and obtain that $\delta\ge (N_F^*)^2(1-k^2(X-k)^{-1})$. Invoking this bound in (\ref{2021-06-14}) we obtain 
\begin{displaymath}
{\bar \delta} \le \E^*(\Delta_F^*)^2+(N_F^*)^2k^2(X-k)^{-1}.
\end{displaymath}  
Hence the ratio in the exponent of (\ref{Janson}) satisfies
\begin{equation}
\label{2021-06-15+1}
\frac{(N_F^*)^2}{{\bar\delta}} 
\ge 
\frac{(N_F^*)^2} 
{2
\max\{\E^*(\Delta_F^*)^2, (N_F^*)^2k^2(X-k)^{-1}\}
}
 = 
 \frac12 \min\left\{\frac{(N_F^*)^2} {\E^*(\Delta_F^*)^2},\frac{X-k}{k^2}\right\}.
\end{equation}
We will show below that there exists $c_k>0$ (independent of $t$)
such that $N_F^*>t$ implies 
\begin{equation}\label{2021-06-15+2}
\frac{(N_F^*)^2} {\E^*(\Delta_F^*)^2}>c_kt^{2/k}.
\end{equation}
We also note that $N_F^*>t$ implies $X>(t/a_F)^{1/k}$ 
(we use 
$a_F\binom{X}{k}
\ge 
a_F\binom{X}{k}Q^{e_F}=N_F^*$). Therefore, 
on the event $N_F^*>t$
the right side of (\ref{2021-06-15+1}) is at least 
\begin{equation}
\label{2022-02-16}
\frac{1}{2}
\min
\left\{
c_kt^{2/k}, \,\frac{ (t/a_F)^{1/k}-k}{k^2}
\right\}
\end{equation}
 and this quantity scales as $t^{1/k}$ as $t\to+\infty$.
 Finally, from (\ref{Janson}), (\ref{2021-06-15+1}), 
 (\ref{2022-02-16}) we obtain that on the event 
 $N_F^*>t$,
\begin{displaymath}
p_{\varepsilon}^* 
\le 
e^{-\varepsilon^2\Theta(t^{1/k})} 
= 
o(t^{-\alpha})\qquad {\text{as}} \qquad t\to+\infty.
\end{displaymath}
 We conclude that $P_2'=o(t^{-\alpha})$.
  It remains to show (\ref{2021-06-15+2}). 
 We observe that inequalities $N_F^*\le a_F\Psi_F(X,Q)$ and $N_F^*>t$ imply  $\Psi_F(X,Q)>t/a_F>1$, where the last inequality holds for $t>a_F$.
 Then (\ref{2021-06-21}) 
 implies $\Phi_F(X,Q)\ge(\Psi_F(X,Q))^{2/k}$ and
  (\ref{2021-06-21+3}) implies
\begin{displaymath}
\frac{(N_F^*)^2}{\E^*(\Delta_F^*)^2} 
\asymp 
\frac{\Phi_F(X,Q)}{1-Q}
\ge 
\Phi_F(X,Q)
\ge
\Psi_F^{2/k}(X,Q)
\ge 
(t/a_F)^{2/k}.
\end{displaymath}
\medskip

{\it Proof of} $P_3=o(t^{-\alpha})$.
In the proof we  apply exponential inequalities for upper tails of 
subgraph counts in Bernoulli random graphs \cite{JansonORucinski}. 
For reader's convenience, we state the result of 
\cite{JansonORucinski} we will use. Let $\Delta_F$ be 
the maximum degree of $F$. Let
\begin{displaymath}
M_F(n,p) 
= 
\begin{cases}
1 
\qquad \qquad \qquad \qquad \qquad \qquad 
{\text{if}} 
\quad 
p<n^{-1/m_F}, 
\\
\min_{H\subset F}\bigl(\Psi_H(n,p)\bigr)^{1/\alpha_H^*} 
\qquad \, 
{\text{if}} 
\quad 
n^{-1/m_F}
\le 
p\le n^{-1/\Delta_F}, 
\\
n^2p^{\Delta_F} 
\qquad \qquad \quad \qquad \qquad \ \ \, 
{\text{if}} 
\quad 
p
\ge 
n^{-1/\Delta_F}.
\end{cases}
\end{displaymath}
Here $\alpha_H^*$ is the fractional independence number of 
 a graph $H$, see \cite{JansonORucinski}. 
We do not define the fractional independence number 
here as we only use  the upper bound $\alpha_H^*\le v_H-1$ that holds for any $H$ 
with $e_H>0$, see formula (A.1) in \cite{JansonORucinski}. 
Let $\xi_F$ be the number of copies of $F$ in $G(n,p)$. By 
Theorems 1.2 and 1.5 of \cite{JansonORucinski}, for any $\eta>0$ 
there exists 
 $c_{\eta, F}>0$
 such that uniformly in $p$ and $n\ge k$ 
(recall that $k=v_F$ is the number of vertices of $F$) 
we have
\begin{equation}\label{JOR}
\PP\{\xi_F\ge (1+\eta)\E \xi_F\} \le e^{-c_{\eta, F}M_F(n,p)}.
\end{equation}
We will  apply (\ref{JOR}) to the number $N_F$ of copies of $F$ in $G(X,Q)$ conditionally given $X,Q$, see (\ref{2021-06-22+1}) below.

We write, for short, $s=\varepsilon t$ and estimate 
$P_3\le \PP\{\Delta_F^*>s\}$. 
Let $\eta>0$. We split
\begin{displaymath}
\PP\{\Delta_F^*>s\} =
\PP\{\Delta_F^*>\eta N_F^*, \Delta_F^*>s\} + 
\PP\{\Delta_F^*\le\eta N_F^*, \Delta_F^*>s\}
=:P_{31}+P_{32}
\end{displaymath}
and estimate the  probabilities $P_{31}$ and $P_{32}$ separately.
The second probability
\begin{equation}\label{2021-06-22+3}
P_{32} \le \PP\{N_F^*>s/\eta\}=\eta^{\alpha}(a+o(1))s^{-\alpha}
\end{equation}
can be made negligibly small by choosing $\eta$ arbitrarily small.

Now we upper  bound the 
 remaining probability
$P_{31}$. Introduce the events
\begin{displaymath}
{\cal A}_1=\bigl\{Q\le X^{-1/m_F}\bigr\},
\qquad \! \! \!
{\cal A}_{21} = \bigl\{X^{-1/m_F}< Q< X^{-1/\Delta_F}\bigr\},
\qquad \! \! \!
{\cal A}_{22} = \bigl\{Q\ge X^{-1/\Delta_F}\bigr\},
\end{displaymath} 
and put ${\cal A}_2={\cal A}_{21}\cup{\cal A}_{22}$ 
(note that $\Delta_F\ge 2m_F= 2e_F/v_F$).
 We split 
\begin{equation} 
P_{31} = {\tilde P}_1+{\tilde P}_2, \qquad
{\tilde P}_i 
:
= \PP\{\Delta_F^*>\eta N_F^*, \Delta_F^*>s, {\cal A}_i\}
\end{equation}
and  estimate ${\tilde P}_1$ and ${\tilde P}_2$ 
 separately.
We firstly consider  ${\tilde P}_1$. The inequality 
$Q\le X^{-1/m_F}$ implies $\Psi_F(X,Q)\le 1$. Consequently,
(\ref{2021-06-21}) implies $\Phi_F(X,Q)\ge \Psi_F(X,Q)$.
The latter inequality together with  (\ref{2021-06-21+3}) imply 
$\E^* (\Delta_F^*)^2 \le c_k\Psi_F(X, Q)\le c_k$ for some $c_k>0$. Hence, on the event ${\cal A}_1$ we have $\E^* (\Delta_F^*)^2\le c_k$. Finally, by Markov's inequality,
\begin{equation}\label{2021-06-22+4}
{\tilde P}_1 \le \PP\{\Delta_F^*>s, {\cal A}_1\} = 
\E \bigl({\mathbb I}_{{\cal A}_1} \E^*{\mathbb I}_{\{\Delta_F^*>s\}} \bigr) \le
\E \bigl( {\mathbb I}_{{\cal A}_1}\E^*(\Delta^*_F)^2s^{-2}\bigr) \le 
c_ks^{-2}.
\end{equation}

We secondly consider ${\tilde P}_2$. The inequality 
$X^{-1/m_F}< Q$ implies $\Psi_F(X,Q)>1$. For 
balanced $F$ this yields $\Psi_H(X,Q)>1$ for every 
$H\subset F$ with $e_H>0$. Then, 
by using $\alpha^*_H\le v_H-1$ we obtain 
\begin{displaymath}
\min_{H\subset F :\, e_H>0} 
\bigl(\Psi_H(X,Q)\bigr)^{1/\alpha_H^*}
\ge 
\min_{H\subset F:\, e_H>0} \bigl(\Psi_H(X,Q)\bigr)^{1/v_H} 
= 
\bigl(\Psi_F(X,Q)\bigr)^{1/v_F}.
\end{displaymath}
In the last step we used the fact that $F$ is balanced  once again.
Hence, on the event ${\cal A}_{21}$ we have (recall that $v_F=k$)
\begin{equation}\label{2021-06-22}
M_F(X,Q)\ge \bigl(\Psi_F(X,Q)\bigr)^{1/k}.
\end{equation} 
We observe that (\ref{2021-06-22}) holds
on the event ${\cal A}_{22}$ as well. Indeed, the inequality 
$Q\ge X^{-1/\Delta_F}$ yields 
$M_F(X,Q)\ge X^2Q^{\Delta_F}\ge X$. 
Now the  inequality $X^{v_F}\ge \Psi_F(X,Q)$ 
implies (\ref{2021-06-22}). 

From (\ref{JOR}) and (\ref{2021-06-22}) we obtain the exponential bound
\begin{equation}\label{2021-06-22+1}
\PP^*\{\Delta_F^*>\eta N_F^*\} 
\le  
e^{-c_{\eta, F}M_F(X,Q)} 
\le
e^{-c_{\eta, F}\bigl(\Psi_F(X,Q)\bigr)^{1/k}}.
\end{equation}

Let us bound ${\tilde P}_2$ from above. We fix a (large) number $B>0$ and introduce the events
\begin{displaymath}
{\cal B}_1=\{\Psi_F(X_1,Q_1)>B\ln^k s\}, \qquad 
{\cal B}_2=\{\Psi_F(X_1,Q_1)\le B\ln^k s\}.
\end{displaymath}
We then split
\begin{displaymath}
{\tilde P}_2 = {\tilde P}_{21} + {\tilde P}_{22}, \qquad
{\tilde P}_{2i} = 
\PP\{\Delta^*_F>\eta N_F^*, \Delta_F^*>s, {\cal A}_2,{\cal B}_i\},
\end{displaymath}
and  bound ${\tilde P}_{21}$ from above, by using (\ref{2021-06-22+1}),
\begin{eqnarray}
\nonumber
{\tilde P}_{21} &\le& \PP\{\Delta_F^*>\eta N_F^*, {\cal A}_2, {\cal B}_1\}
=
\E \left( {\mathbb I}_{{\cal B}_1}{\mathbb I}_{{\cal A}_2} \PP^*\{\Delta_F^*>\eta N_F^*\}\right) \\
\label{2021-06-22+5}
&\le&
\E \left({\mathbb I}_{{\cal B}_1} e^{-c_{\eta,F}(\Psi_H(X_1,Q_1))^{1/k}} \right) \le 
e^{-c_{\eta,F}B^{1/k}\ln s}.
\end{eqnarray}

It remains to upper bound ${\tilde P}_{22}$. 
The inequality $\Psi_F(X,Q)>1$, which holds on the event 
${\cal A}_2$, implies (see (\ref{2021-06-21}))
$\Phi_F(X,Q)\ge (\Psi_F(X,Q))^{2/k}$. Furthermore,   (\ref{2021-06-21+3}) implies
\begin{displaymath}
\E^* (\Delta_F^*)^2 
\le 
c_F
\bigl( \Psi_F(X,Q) \bigr)^{2-(2/k)}(1-Q),
\end{displaymath}
where $c_F>0$ depends only on $F$.
Note that on the event ${\cal B}_2$ the right side is upper bounded by 
$c_F(B\ln^k s)^{2-(2/k)}$.
Hence, by Markov's inequality,
\begin{displaymath}
\PP^*(\Delta_F^*>s) 
\le 
s^{-2}\E^*(\Delta_F^*)^2
\le 
c_FB^{2-(2/k)}s^{-2}\ln^{k-2}s.
\end{displaymath}
Finally, we obtain
\begin{eqnarray}
\label{2021-06-22+6}
{\tilde P}_{22} 
\le 
\PP\{\Delta_F^*>s, {\cal A}_2, {\cal B}_2\} 
=
\E \bigl( {\mathbb I}_{{\cal A}_2}
{\mathbb I}_{{\cal B}_2} 
\PP^*\{\Delta_F^*>s\} \bigr)
\le 
c_FB^{2-(2/k)}s^{-2}\ln^{2k-2}s.
\end{eqnarray}

\smallskip

We complete the proof by showing  that for any  $0<\varepsilon<1$ the probability $P_3$, which depends on $\varepsilon$, satisfies  $P_3=o(t^{-\alpha})$ as $t\to+\infty$. 
 Recall that $s=\varepsilon t$. We have  
for any $\eta>0$
\begin{eqnarray}
\nonumber
&&
\limsup_{t\to+\infty}t^{\alpha}P_3 \le 
\limsup_{t\to+\infty} t^{\alpha}\PP\{\Delta_F^*>\varepsilon t\} =
\varepsilon^{-\alpha} \limsup_{s\to+\infty} s^{\alpha} \PP\{\Delta_F^*>s\} \\
\label{2022-02-17++}
&&
\le \varepsilon^{-\alpha} \limsup_{s\to+\infty} s^{\alpha}
({\tilde P}_1+{\tilde P}_{21}+{\tilde P}_{22}+P_{32}) \le 
(\eta/\varepsilon)^{\alpha}a.
\end{eqnarray}
Hence $\limsup_{t\to+\infty}t^{\alpha}P_3=0$.
The last inequality of (\ref{2022-02-17++}) follows from 
(\ref{2021-06-22+3}), (\ref{2021-06-22+4}), (\ref{2021-06-22+5}), and (\ref{2021-06-22+6}). Indeed, given $\eta>0$, we choose $B=B(\eta)$ (in (\ref{2021-06-22+5}), (\ref{2021-06-22+6})) large enough so that $c_{\eta,F}B^{1/k}>2$. Then ${\tilde P}_{21}\le s^{-2}$ and $\limsup_{s}s^{\alpha}{\tilde P}_{21}=0$. We also mention the obvious relations $\limsup_{s}s^{\alpha}{\tilde P}_{1}=0$ and $\limsup_{s}s^{\alpha}{\tilde P}_{22}=~0$.
\end{proof}

\end{document}